\documentclass[english]{ourlematema}
\usepackage{amsmath, amsthm, amssymb, hyperref, color}
\usepackage{MnSymbol} % for \dashedrightarrow command
\usepackage{graphicx}
\usepackage{caption}
\usepackage[all]{xypic}
\usepackage{verbatim}
\usepackage{chemfig, chemnum}
\usepackage{tikz}
\usepackage[linesnumbered,lined,commentsnumbered]{algorithm2e}
\usepackage{caption}
\usepackage{subcaption}
\usepackage{float}
\usepackage{makecell}
\usepackage{array}
\usepackage{fancyvrb}

\newtheorem{theorem}{Theorem}
\numberwithin{theorem}{section}
\newtheorem{proposition}[theorem]{Proposition}

\newtheorem{corollary}[theorem]{Corollary}

\newtheorem{remark}[theorem]{Remark}
\newtheorem{example}[theorem]{Example}
\newtheorem{conjecture}[theorem]{Conjecture}

\theoremstyle{definition}

\newcommand{\PP}{\mathbb{P}}
\newcommand{\RR}{\mathbb{R}}
\newcommand{\QQ}{\mathbb{Q}}
\newcommand{\CC}{\mathbb{C}}

\newcommand{\NN}{\mathbb{N}}

 \title{Tangent Quadrics in Real 3-Space}
 
  \author{Taylor Brysiewicz}
  \address{%
  Department of Applied and Computational Mathematics and Statistics, University of Notre Dame \\
\email{tbrysiew@nd.edu}
}

  \author{Claudia Fevola}
  \address{%
  MPI for Mathematics in the Sciences, Leipzig \\
\email{Claudia.Fevola@mis.mpg.de}
}

\author{Bernd Sturmfels}
\address{%
  MPI for Mathematics in the Sciences, Leipzig \\
\email{bernd@mis.mpg.edu }
}

 \date{2020/10/11}

\begin{document}
\maketitle
\begin{abstract}
\noindent 
We examine quadratic surfaces in $3$-space that are
tangent to nine given figures. These figures can be points, lines, planes or quadrics.
The numbers of tangent quadrics were determined by
Hermann Schubert in 1879. We study the associated systems of polynomial equations,
also in the space of complete quadrics,
and we solve them using certified numerical methods.
Our aim is to show that  Schubert's problems are fully real. 
\end{abstract}

\section{Introduction}

There are $3264$ conics tangent to five given conics in
the projective plane $\PP^2$, and there exist five explicit conics
so that all $3264$ complex solutions are real \cite{notices, ronga}.
We here study such tangency questions in
one dimension higher. We consider quadrics
(i.e.~quadratic surfaces) in $\PP^3$.  Schubert \cite{schubert}
found that there are $666841088$ quadrics tangent to nine given 
quadrics in $\PP^3$. Our ultimate goal is to decide whether
there exist  nine real quadrics so that all complex solutions are real.
In this article we present first steps towards answering that question.

\begin{figure}[ht]
  \centering
  \includegraphics[scale=0.55]{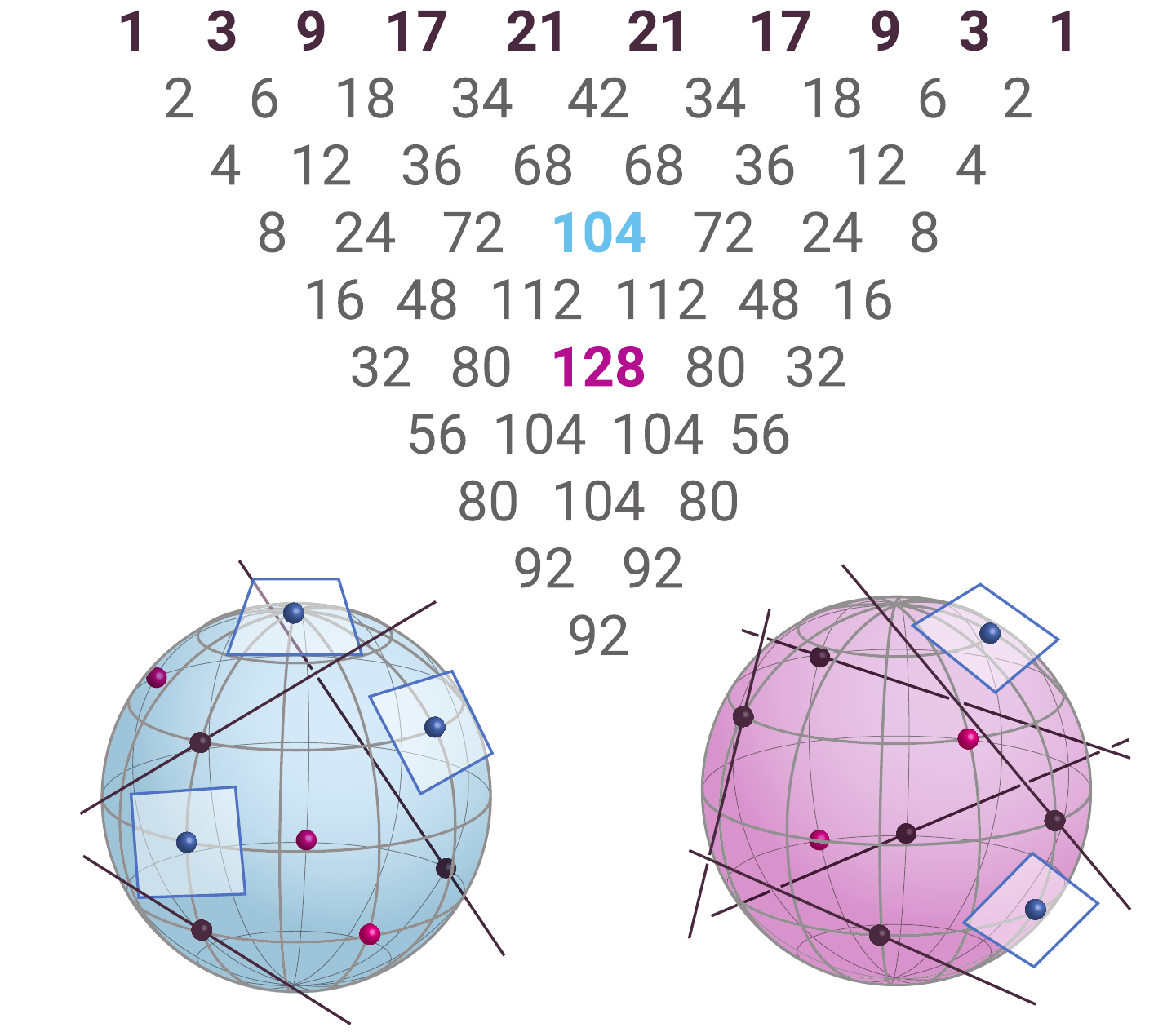} \vspace{-0.09in}
  \caption{Schubert's triangle for tangency of quadrics in 3-space.}
  \label{fig:triangle}
\end{figure}

Our study centers around {\em Schubert's triangle} which is displayed
in Figure \ref{fig:triangle}. For each triple $(\alpha,\beta,\gamma) \in \NN^3$
with $\alpha+\beta+ \gamma =9$, the triangle shows the number $p^\alpha \ell^\beta h^\gamma$
of quadrics that pass through $\alpha$ given points, are tangent to $\beta$ given lines,
and are tangent to $\gamma$ given planes. The two pictures, in blue and red, illustrate the 
geometric meaning of the entries
$p^3 \ell^3 h^3 = {\color{blue} 104}$ and $p^2 \ell^5 h^2 = {\color{red} 128}$.
% We shall present explicit instances such that all
% {\color{blue} 104} resp.~{\color{red} 128} quadrics are real.

Schubert derives these numbers in
\cite[\S 22]{schubert}. In
\cite[page 106]{schubert} he argues as follows. Quadrics degenerate into 
complete flags, consisting of a point on a line in a plane in $\PP^3$.
 Such a flag counts  with multiplicity two, 
 since $q = 2 (p+\ell+h)$,
 by Proposition \ref{prop:hurwitz}.
  We seek quadrics that satisfy
one of the three tangency conditions, for each of the nine given flags. The number of
such quadrics equals
\begin{equation} \label{eq:q^9}
q^9  = \,\, 2^9 \!\!\!\!\! \sum_{\alpha{+}\beta{+}\gamma=9} \begin{matrix}
 \frac{9!}{ \alpha ! \beta ! \gamma !} \end{matrix} \cdot p^\alpha \ell^\beta h^\gamma
\,\, = \,\, 2^9 (\, \cdots + 1680 \cdot 
{\color{blue} 104} +  \cdots + 756 \cdot {\color{red}128} + \cdots \,).
 \vspace{-0.07in}
\end{equation}
The second equation is  in the cohomology ring of the space of complete quadrics.
In (\ref{eq:q^9})  we multiply each entry in Schubert's triangle with the corresponding
trinomial coefficient $ \frac{9!}{\alpha ! \beta! \gamma!}$, we add up the products, we
multiply the sum by $2^9 = 512$, and we obtain
 $q^9 = 666841088$.
This derivation is the analogue in $\PP^3$ of the pentagon count 
for the $2^5 (p+\ell)^5 = 3264$ conics in \cite[Figure~3]{notices}.

Schubert's calculus predicts the number of complex solutions to a system of 
polynomial equations that depend on geometric figures like lines and planes in $\PP^3$.
In this article we study these polynomial equations and present
practical tools for solving them. Our main interest is in solutions over 
the real numbers~$\RR$.

This paper is organized as follows. In Section \ref{sec2}
we introduce coordinates for points, lines, planes and quadrics,
and derive the polynomials that describe our tangency conditions.
Section \ref{sec3} is dedicated to the space of complete quadrics, a variety
in $\PP^9 \times \PP^{20} \times \PP^9$. We determine its prime ideal,
we recover Schubert's triangle as its multidegree, 
and we write our tangency conditions in that setting.

In Section \ref{sec4} we argue that Schubert's triangle is mostly real.
We present explicit instances where all tangent quadrics are real.
These instances were found by substantial computations using the
software {\tt HomotopyContinuation.jl} \cite{julia}. Our computations
and the certification process are described in Section~\ref{sec5}.

In Section \ref{sec6} we turn to Schubert's pyramid. It gives the numbers
 $p^\alpha \ell^\beta h^\gamma q^\delta$
of quadrics through $\alpha $ points that are
tangent to $\beta$ lines, $\gamma$ planes and $\delta$ quadrics;
see Figure~\ref{fig:levels}.
At the top of this pyramid lives
 $q^9 = 666841088$.
We discuss the associated polynomial systems
and we state two conjectures about their~reality.

\section{Coordinates and Equations}
\label{sec2}

We begin with the coordinates that describe
our geometric figures. A point $P$ in $\PP^3$ 
is represented by a vector $p = (p_1,p_2,p_3,p_4)$.
A line can be given by~a $2 \times 4$ matrix $L$,
and a plane by a $3 \times 4$ matrix $H$.
We often use Pl\"ucker coordinates
$$  
 \vspace{-0.04in}
\ell = (\ell_{12},\ell_{13},\ell_{14}, \ell_{23}, \ell_{24}, \ell_{34}) \quad {\rm and} \quad
h = (h_{234}, -h_{134}, h_{124}, -h_{123}). $$
Here $\ell_{ij}$ is the $2 \times 2$ minor of $L$
with column indices $i$ and $j$.
Note the {\em Pl\"ucker relation} $\,\ell_{12} \ell_{34} - \ell_{13} \ell_{24} + \ell_{14} \ell_{23} = 0$.
 Likewise
$h_{ijk}$ denote the $3 \times 3$ minors of $H$.

\begin{remark}  \label{rmk:flag}
Inclusion relations are written in Pl\"ucker coordinates as follows:
$$ \begin{matrix}
P \subset H: & &  
p_1 h_{234} - p_2 h_{134} + p_3 h_{124} - p_4 h_{123} , \smallskip \\
P \subset L: & &  \begin{matrix} 
p_1 \ell_{23} - p_2 \ell_{13} + p_3 \ell_{12} \, ,\,\,
p_1 \ell_{24} - p_2 \ell_{14} + p_4 \ell_{12}, \\
p_1 \ell_{34} - p_3 \ell_{14} + p_4 \ell_{13}\,,\,
p_2 \ell_{34} - p_3 \ell_{24} + p_4 \ell_{23}, \end{matrix}  \smallskip \\
L \subset H: & & \begin{matrix}
 \ell_{12} h_{134} - \ell_{13} h_{124} + \ell_{14} h_{123}\,,\,\,
\ell_{12} h_{234} - \ell_{23} h_{124} + \ell_{24} h_{123},\\
\ell_{13} h_{234} - \ell_{23} h_{134} + \ell_{34} h_{123}\,,\,\,
\ell_{14} h_{234} - \ell_{24} h_{134} + \ell_{34} h_{124}. \end{matrix}
\end{matrix}
$$
A triple $(P,L,H)$ satisfying $P \subset L \subset H$ is called a {\em complete flag}.
The variety of complete flags is irreducible of dimension six
in $\PP^3 \times \PP^5 \times \PP^3$.
The prime ideal of this {\em flag variety} is generated by the
nine quadrics above, together with the Pl\"ucker relation.
These ten generators form a Gr\"obner basis
\cite[Theorem~14.6]{cca}.
\end{remark}

Each quadric in $\PP^3$ is represented by a symmetric $4 \times 4$ matrix  $X = (x_{ij})$.
The point $P$ lies on the quadric $X$ if $ P X P^T = 0$.
Similarly, the condition for $X$ to be tangent to a line $L$ or to a plane $H$ is
given by the vanishing of the polynomial
\begin{equation}
\label{eq:wedge23}
{\rm det}(L X L^T) \,= \, \ell (\wedge_2 X) \ell^T   \quad {\rm or} \quad
{\rm det}(H X H^T) \,= \, h (\wedge_3 X) h^T  .
\end{equation}
Here $\wedge_i X$ denotes the $i$-th exterior power of the $4 \times 4$ matrix $X$.
The entries of $\wedge_i X$ are  the $i \times i$ minors of $X$.
The rows and columns are
labeled so that (\ref{eq:wedge23}) holds.

Suppose we are given $\alpha$ points $P_i$,
$\beta $ lines $L_j$, and $\gamma$ planes $H_k$, all generic,
where $\alpha + \beta + \gamma = 9$.
We wish to solve these nine homogeneous equations for $X$:
\begin{equation} \label{eq:schub}
\! P_i X P_i^T \,= \,
{\rm det}(L_j X L_j^T)  \,=\,
{\rm det}(H_k X H_k^T)  \,= \, 0
\,\,\,
{\rm for} \,\,
1 \leq i \leq \alpha,\,
1 \leq j \leq \beta,\,
1 \leq k \leq \gamma.
\end{equation}
Here $X$ is an unknown symmetric $4 \times 4$ matrix,
viewed as a point in $\PP^9$, that satisfies ${\rm det}(X) \not= 0$.
 B\'ezout's Theorem suggests that the number of 
complex solutions to (\ref{eq:schub}) equals $1^\alpha 2^\beta 3^\gamma$.
This number is correct when $\alpha \geq 4$ and $\gamma \leq  2$.
In all other cases,
 the equations (\ref{eq:schub}) have extraneous solutions
that are removed by saturation with respect to the  ideal
$\langle {\rm det}(X) \rangle$.
This saturation step~can be carried out in {\tt Macaulay2} \cite{M2}.
For each choice of $(\alpha,\beta,\gamma)$,
we obtain a Gr\"obner basis that reveals the number of
solutions in $\PP^9$. This computation proves the correctness
of Schubert's triangle. For solving  (\ref{eq:schub})
numerically, see Section~\ref{sec5}.

\smallskip

We next discuss the condition for $X$ to be tangent to a fixed quadric $U= (u_{ij})$.
\begin{lemma} \label{lem:bigdisc}
The condition for two quadrics $U$ and $X$ to be tangent  in $\PP^3$
is given by the discriminant of the  quartic
$f(t) = {\rm det}(U + t X)$. This is an 
 irreducible polynomial with
$67753552$ terms of degree $(12,12)$ in the $20$ unknowns $u_{ij},\,x_{ij}$.
 \end{lemma}

\begin{proof}
The tangency condition means that the intersection curve of the
 quadrics $U$ and $X$ is singular in $\PP^3$.
By the Cayley trick of elimination theory \cite[\S 3.2.D]{GKZ}, this is singular if and only if
the line spanned by $U$ and $X$ in $\PP^9$ is tangent to
the hypersurface $\{{\rm det}(X) = 0\}$. That condition is given by the
discriminant of $f(t)$, which is  known
as the {\em Hurwitz form} of $\{{\rm det}(X) = 0\}$. 
We found~its expansion into $67753552$ monomials
 with the computer algebra system {\tt Maple}.
\end{proof}

We denote the above discriminant by $\Sigma(U,X)$.
If $U$ is a symmetric matrix with random entries in $\RR$ or $\CC$ then
$\Sigma(U,X)$ is a polynomial of degree $12$ in
ten unknowns $x_{ij}$ with $241592$ terms.
Given nine quadrics $U_1,\ldots,U_9$ in $\PP^3$, the quadrics
tangent to these solve the following equations in $\PP^9$:
\begin{equation}
\label{eq:hardsystem}
 \Sigma(U_1,X) \, = \,
\Sigma(U_2,X) \, = \, \,\cdots =\, \,\Sigma(U_9,X) \, = \, 0 \quad {\rm and} \quad
{\rm det}(X) \not= 0. 
\end{equation}
B\'ezout's Theorem suggests that the nine equations have $12^9$
complex solutions, but the inequation decreases that number to
  $q^9 = 666841088$.
We derived  this in the Introduction from
Figure~\ref{fig:triangle}. A key ingredient was the identity
$q = 2 (p+\ell+h)$.

We next prove this identity by an explicit geometric degeneration.
Let $V$ be an invertible real $4 \times 4$ matrix,
and let $P \subset L \subset H$ be the flag given by its first
three rows. We introduce a parameter $\epsilon > 0$,
and we consider the quadric defined by 
\begin{equation}
\label{eq:deformation}
 U_\epsilon \,\,=\,\, V^{-1} \cdot {\rm diag}(\epsilon^3,\epsilon^2,\epsilon,1) \cdot (V^{-1})^T . 
 \end{equation}
We investigate the behavior of the tangency condition for $U_\epsilon$
and $X$, as $\epsilon \rightarrow 0$.

\begin{proposition} \label{prop:hurwitz}
The leading form in $\epsilon$ of the specialized
Hurwitz form equals
\begin{equation}
\label{eq:hurwitzfac}
\Sigma(U_\epsilon , X) \,\, = \,\, 
(P X P^T)^2 \cdot\, {\rm det}(L X L^T)^2 \cdot \, {\rm det}(HX H^T)^2 \cdot \epsilon^8 \,\,+\,\,
\hbox{higher terms in $\epsilon$}.
\end{equation}
This implies the identity 
$q = 2 (p+\ell+h)$ in the appropriate cohomology ring.
\end{proposition}

\begin{proof} 
The factorization in (\ref{eq:hurwitzfac})
can be seen directly from the discriminant of
$$ f(t) \,=\, 
 {\rm det}(U_\epsilon + t X) \, =\,
c_0 +c_1 t + c_2 t^2 + c_3 t^3 + c_4 t^4 . $$
The coefficients $c_i$ are polynomials in $\epsilon$
with orders of vanishing $6,3,1,0,0$ at $ \epsilon = 0$. The 
discriminant has vanishing order $8$ at $\epsilon=0$, 
and this order is uniquely attained by its
monomial $c_1^2 c_2^2 c_3^2$. The  factors $c_1,c_2,c_3$
map to those in (\ref{eq:hurwitzfac}).
\end{proof}

\section{Complete Quadrics}
\label{sec3}

A geometric setting for our tangency problems is
the {\em space of complete quadrics}. By definition, this is the variety obtained
 as the closure of the image of the~map
\begin{equation}
\label{eq:464}  \PP^9 \,\,\dashrightarrow\,\,
 \PP^9 \times   \PP^{20} \times    \PP^9 \, ,\,\,\, \quad
 X \,\,\mapsto \,\, (X , \,\wedge_2 X, \,\wedge_3 X) \,\,\,=: \,\,\,(X,Y,Z) . 
 \end{equation}
 Here $X= (x_{ij})$ and $Z = (z_{ijk,lmn})$ are symmetric $4 \times 4$ matrices
 and $Y = (y_{ij,kl})$ is a symmetric $6 \times 6$ matrix.
 The rows and columns of $Y$ and $Z$ are indexed just like the entries of $\ell$ and $h$.
  The $\NN^3$-homogeneous
 ideal $\mathcal{I}_4$ of that $9$-dimensional variety lives in the
polynomial ring $\QQ[X,Y,Z]$ in $10+21+10=41$ unknowns.

\begin{theorem} \label{thm:CQ}
The space of complete quadrics is a smooth variety of dimension nine.
Its prime ideal $\mathcal{I}_4$ is minimally generated by $164$ polynomials, namely \\
$\bullet \,$  one linear form of degree $(010)$, i.e. $y_{12,34} - y_{13,24} + y_{14,23}$, \\
$\bullet \,$ $20$  quadrics of degree $(020)$, e.g.~$y_{12,24} y_{24,34} - y_{13,24} y_{24,24} + y_{14,24} y_{23,24}$, \\
$\bullet $  $15$ quadrics of degree $(\!101\!)$, e.g.~$x_{11} z_{123,234} - x_{12} z_{123,134} + x_{13} z_{123,124} - x_{14} z_{123,123}$, \\
$\bullet \,$  $64$ quadrics of degree $(011)$, e.g.~$y_{12,13} z_{123,134} - y_{13,13} z_{123,124} + y_{13,14} z_{123,123}$,  \\
$\bullet \,$  $64$ quadrics of degree $(110)$, e.g.~$\,x_{11} y_{12,23} - x_{12} y_{12,13} + x_{13} y_{12,12}$. \\
 Schubert's triangle in Figure~\ref{fig:triangle} equals the
{\em multidegree} of $\,\mathcal{I}_4\,$ 
in the $\NN^3$-grading.
\end{theorem}

\begin{proof}[Proof]
The closure of the image of (\ref{eq:464}) is irreducible of dimension nine since
$X$ appears in the first coordinate. The smoothness of this variety is well-known
in the theory of spherical varieties.
For a new perspective and proof see \cite[\S 3.C]{MMW}. 

The $164$ polynomials  were found
by  computation using {\tt Macaulay2} \cite{M2}.  To show that they generate the prime ideal $\mathcal{I}_4$,
we use  \cite[Proposition~23]{GSS} inductively.
 We eliminate one variable
that occurs linearly in some equation and is not a zero-divisor modulo the current ideal.
After checking these hypotheses, we replace the ideal by the elimination ideal,
which is prime by induction.
This process was found to work for various natural orderings of the entries in $X,Y,Z$.

The multidegree is a standard construction for multigraded commutative rings
\cite[Section 8.5]{cca}. For a variety in a product of projective spaces,
it is the class of that variety in the cohomology ring of the ambient space.
The built-in command {\tt multidegree}
 in {\tt Macaulay2} takes only a few seconds to find
 the multidegree from our $164$ polynomials. The output of this 
 {\tt Macaulay2} computation 
   is a ternary form 
   in the unknowns $T_0, T_1, T_2$. It has
   $55$ terms of degree  ${\rm codim}(\mathcal{I}_4) = 29$.
 The coefficient of $T_0^{9-\alpha} T_1^{20-\beta} T_2^{9-\gamma}$
is the number $p^\alpha \ell^\beta h^\gamma$  in  Figure~\ref{fig:triangle}.
   This  computation is an {\it ab initio} derivation of Schubert's triangle.
    \end{proof}

The variety $V(\mathcal{I}_4)$ 
captures degenerations of quadrics that matter in
intersection theory \cite{MMW}. We saw this in
Proposition \ref{prop:hurwitz} where the quadric becomes a flag
$P \subset L \subset H$. The relationship to the
flag variety is made precise as follows:

\begin{corollary}
The variety of complete flags in $\PP^3$
is the inverse image of
$V(\mathcal{I}_4)$  under the componentwise
 Veronese embedding
$ \,\PP^3 \times \PP^5 \times \PP^3
\,\hookrightarrow\, \PP^9 \times \PP^{20} \times \PP^9$.
\end{corollary}

\begin{proof}
The Veronese map takes $(p,\ell,h)$ to the
 rank one matrices $(X,Y,Z) = (p^T p,  \,\ell^T \ell , \, h^T h)$.
Substituting this into $\mathcal{I}_4$ and saturating by
the irrelevant ideal of 
$ \PP^3 \times \PP^5 \times \PP^3$ yields the Gr\"obner basis
for the flag variety in Remark 
\ref{rmk:flag}.
\end{proof}

We next lift our tangency conditions 
from the space $\PP^9 $ of symmetric matrices $X$
to the space of complete quadrics in $\PP^9 \times \PP^{20} \times \PP^9$.
We write $\,\mathcal{B} \,= \, \langle X \rangle \, \cap \, \langle Y \rangle \,\cap \,\langle Z \rangle\,$
for the irrelevant ideal of that product of projective spaces.

The condition that a quadric contains a point $p$ is the linear form $p X p^T$
in the unknown $X$. Similarly, tangency to a line $\ell$ is the linear form $\ell \, Y \ell^T$
in the unknown $Y$, and 
tangency to a plane $h$ is the linear form $h Z h^T$ in the unknown $Z$.
Without loss of generality, we can assume that one
given figure is a coordinate subspace in $\PP^3$. Then the three
linear forms are variables $x_{11}$, 
$y_{12,12}$ or $ z_{123,123}$.

However, if we augment $\mathcal{I}_4$ by one such variable
then the resulting ideal is not prime. To get the correct prime
ideal we must saturate by the irrelevant ideal $\mathcal{B}$.
We first summarize what happens when we 
add the constraint for a point.
The result is the same for the plane constraint
if we swap the roles of $X$ and $Z$.

\begin{proposition} \label{prop:pointadd}
The saturation $\bigl((\mathcal{I}_4 + \langle x_{11} \rangle) : \mathcal{B}^\infty \bigr)$ is
the prime ideal of the variety of complete quadrics that contain a given point. It has
$13$ minimal generators in addition to the $164$ generators of $\mathcal{I}_4$, namely
ten of degree $(020)$ and one each of degree $(100)$, $(003)$ and $(011)$.
The multidegree of this ideal is the triangle of size eight that is
obtained by deleting the lower right edge in  Figure~\ref{fig:triangle}.
\end{proposition}

\begin{proof}
This is proved by a  {\tt Macaulay2} computation.
The new equation of degree $(100)$ is $x_{11}$. The
new equation of degree $(003)$ is the complementary~$3 \times 3$ minor of $Z$. 
Generators of degree $(020)$ arise from Bareiss formula which says that $x_{11}$ times
any $3 \times 3 $ minor of $X$ containing $x_{11}$ equals a $2 \times 2$ minor of $Y$.
\end{proof}

% Here is the analogous result for  complete quadrics that are tangent to a line.

\begin{proposition} \label{prop:lineadd}
The saturation $\bigl((\mathcal{I}_4 + \langle y_{12,12} \rangle) : \mathcal{B}^\infty \bigr)$ is
the prime ideal for the complete quadrics that are tangent to a line.
It has three minimal generators, of degrees
$(010), (200), (002)$,  in addition to the $164$ generators of $\mathcal{I}_4$.
This is one entry of $Y$ and the corresponding  $2 \times 2 $ minors of $X$ and $Z$.
The multidegree  is the triangle of size eight 
obtained by deleting the top edge in~Figure~\ref{fig:triangle}.
\end{proposition}

It would be desirable to extend Theorem \ref{thm:CQ}
to $n \times n$ matrices for $n \geq 5$, i.e.~to identify
minimal generators for the multihomogeneous prime ideal of the
space of complete quadrics. These are  relations among
all minors of a symmetric $n \times n$ matrix that respect the
fine grading coming from the size of the~minors.
Results by Bruns et al.~\cite{BCV} indicate that
relations of degree $\leq 2$ will not suffice.

\section{Schubert's Triangle}
\label{sec4}

At present, we have the following result on the reality of Schubert's triangle.

\begin{theorem} \label{thm:inprogress}
For at least $46$ of the $55$ problems in Schubert's triangle,
there exists an open set of real instances,
consisting of $\alpha$ points, $\beta$ lines and $\gamma$ planes,
such that all complex solutions in $\PP^9$ to
the polynomial equations in (\ref{eq:schub}) are real.
For the other nine problems, the current status is summarized in
Remark~\ref{rmk:notyet}.
\end{theorem}

\begin{example} \label{ex:333}
Fix $(\alpha,\beta,\gamma) = (3,3,3)$. 
We consider the configuration
$$ \begin{matrix} p & =  & 
(1, \frac{439}{922}, -\frac{347}{271}, \frac{67}{343}) \,,\,\,
(1, -\frac{211}{484}, \frac{153}{346} ,\frac{257}{254}) \, , \,\,
(1, -\frac{575}{404}, \frac{131}{320}, -\frac{37}{42}), \medskip  \\
  \ell &  = &
(- \frac{92}{159}, -\frac{92}{293}, \frac{120}{307},  \frac{77}{256}, \frac{76}{391}, \frac{96}{311})\, , \,\,
( \frac{107}{114}, \frac{18}{383}, -\frac{109}{116}, \frac{37}{217}, \frac{45}{307}, \frac{47}{264})\, , \,\,
 \medskip \\ & & 
( -\frac{365}{302}, - \frac{45}{368}, \frac{172}{209}, \frac{74}{245}, \frac{25}{62}, \frac{87}{353}),
\medskip  \\
h  & = &
(\frac{193}{182}, \frac{75}{397}, -\frac{244}{631}, \frac{195}{272} ) \, , \,\,
( \frac{91}{307}, - \frac{17}{122}, - \frac{553}{837}, \frac{70}{309} ) \, , \,\,
( \frac{919}{295}, \frac{103}{36}, \frac{1199}{371}, \frac{57}{176} ).
\end{matrix}
$$
All $104$ complex quadrics tangent to these nine figures are found to be real.
Thus, this is a fully real instance for the scenario shown in blue
 in Figure \ref{fig:triangle}.\end{example}

\begin{remark} \label{rmk:notyet}
Up to the natural involution, given by swapping  points and planes, 
there are $30$ distinct tangency problems in Schubert's triangle.
For five of the problems we have not yet succeeded in verifying reality.
They are as follows:
$$
\begin{matrix}
(\alpha,\beta,\gamma)                         &\!\! &  (3,4,2) & (3,5,1) & (2,6,1) & (1,7,1) & (1,8,0) \\
\! \hbox{Schubert's count over $\CC$} &\!\! &   112     &    80    &  104  &  104 &  92  \\
\hbox{Our current record over $\RR$}  &\!\! &        110      &    74    &    96 &   84   & 84 \\
\end{matrix}
$$
For instance, we know  two points, six lines and a plane in $\PP^3_\RR$ such that
$96$ real quadrics are tangent to these figures. The remaining eight quadrics are complex.
This is derived from  Example \ref{ex:sixliens} by replacing point $P_3$ with a plane.
For the (1,8,0)  case with $84$ real solutions
we use eight tangent lines as in
Example \ref{ex:tangentlines}.
\end{remark}

\begin{proof}[Discussion and proof of Theorem \ref{thm:inprogress}]
All our instances of full reality or maximal reality, along with the
software that certifies correctness, can be found at
\begin{equation} \label{eq:ourwebsite}
\href{https://mathrepo.mis.mpg.de/TangentQuadricsInThreeSpace/index.html}{https://mathrepo.mis.mpg.de}
\end{equation}
For instance, for $(\alpha,\beta,\gamma) = (3,3,3)$,
this website contains the configuration in Example \ref{ex:333},
along with the $104$ tangent quadrics. Each quadric is determined by
its nine points of tangency. The coordinates of these points form a
 $104 \times 9 \times 4$ tensor of floating point numbers in {\tt Julia} format.
The proof of correctness was carried out with the
certification technique in \cite{BRT},
as discussed in Section~\ref{sec5}.
\end{proof}

We now present some  ideas that were helpful
in creating fully real instances. 
Figures given by
 the standard basis $e_1,e_2,e_3,e_4$
 lead to sparse equations in~(\ref{eq:schub}).
 
\begin{example}  \label{ex:sixliens}
The condition for $X$ to be tangent to the six coordinate lines  is
\begin{equation}
\label{eq:4x4principal}
 I \,\,  = \,\, \langle \,x_{ii} x_{jj} - x_{ij}^2 \,: \, 1 \leq i < j \leq 4 \,\rangle. 
 \end{equation}
This is the complete intersection of eight prime ideals, each isomorphic
to the ideal $J$ generated by all $2 \times 2$ minors of $X$.
The eight primes are 
$\, U_{ijk} \star J $, where $\star$ is the Hadamard product, and
$U_{ijk}$ is the $4 \times 4$ matrix with entries
$(-1)^i$, $(-1)^j$ and $(-1)^k$ in positions $(2,3)$, $(2,4)$ and $(3,4)$,
and entries $1$ everywhere else. Seven of these 
{\em scaled Veronese varieties} contain matrices
of rank $3$ or $4$. Their union is defined by the radical ideal $(I:J)$, which has degree $56$.
This is Schubert's number for $\alpha=3, \beta=6, \gamma= 0$. 
We seek three points such that all $56$ quadrics containing these and
 satisfying $I$ are real. One choice that works~is
 $$ P_1 = (1,2,8,7), \quad  P_2 = (1,1,9,2) ,\quad P_3 = (2,5,3,1) . $$
Our six given lines meet pairwise, and they are not generic.
 This leads to 
$48$ of the $56$ quadrics being cones. To get $56$ smooth
quadrics, one perturbs the lines.

We refer to the article  \cite{KW} by Kahle and Wagner for
 a general study of the ideal of principal $2 \times 2$ minors
of a symmetric $n \times n$ matrix of unknowns. 
Their results elucidate the decomposition we found for 
the special case $n=4$ in (\ref{eq:4x4principal}).
\end{example}

\begin{example}
The condition for $X$ to be tangent to the four coordinate planes 
is the ideal generated by the four principal $3 \times 3$ minors.
Saturating by the ideal of all $3 \times 3$ minors yields
a prime ideal $K$ of codimension $4$ and degree $21$. 
This is Schubert's number for $\alpha=5, \beta=0,\gamma=4$.
It is easy to find five points so that all $21$ quadrics 
containing these and satisfying $K$ are real. This instance is generic.

The ideal $K$ is generated by $10$ cubics and $12$ quartics.
The $5$-dimensional variety cut out by $K$ in
 $\PP^9$ has the following nice parametric representation:
$$ \begin{small} X \,= \, \begin{pmatrix} 
x_{12} x_{13} x_{14} \!\!\!\!  & \!\!  x_{12} \!\! & \!\!  x_{13} \!\! & \!\!  x_{14} \\
x_{12} \!\! & \!\! \!\! \! x_{12} x_{23} x_{24}\!\! \!\! & \!\!\!  x_{23} \!\! & \!\!  x_{24} \\
x_{13} \!\! & \!\!  x_{23} \!\! & \!\!\!\! \! x_{13} x_{23} x_{34} \!\! \!\! & \!\!  x_{34} \\
x_{14} \!\! & \!\!  x_{24} \!\! & \!\!  x_{34} \!\! & \!\! \!\! \! x_{14} x_{24} x_{34} \end{pmatrix}
\,\,\,\, {\rm where} \,\,\,\,
{\rm det} \begin{pmatrix}  x_{12} x_{34} \!\! \!\! & \!\!  1 \!\! & \!\!  1 \\ 
1 \!\! & \!\!\!\!  x_{13} x_{24} \!\!\!\! & \!\!  1 \\ 1 \!\! & \!\!  1 \!\! & \!\!\!\!  x_{14} x_{24} \end{pmatrix} = 0.
\end{small}
$$
\end{example}

Our final technique was inspired by the solution to Shapiro's conjecture~\cite{sot}.

\begin{example} \label{ex:tangentlines}
Consider the lines $\ell = (1,2t, 3t^2, t^2, 2t^3,t^4)$ that are tangent
to the twisted cubic curve $\{(1:t:t^2:t^3)\}$. There is a surface
of quadrics tangent to all such lines. We choose
nine nearby lines, by slightly perturbing nine tangent lines.
Our fully real instance for $(\alpha,\beta,\gamma) = (0,9,0)$ was
found in this manner. 
%we already said the below sentence earlier. So let's not say it twice.
%We attempted to transfer this solution to $(\alpha,\beta,\gamma) = (1,8,0)$ by replacing one line by a point on or near that line. However, this did not yet lead to success.
\end{example}

\section{Numerical Methods}
\label{sec5}

We now explain our techniques for
solving the equations (\ref{eq:schub}) and
for certifying the correctness of their solutions.
Each instance is presented in the
Pl\"ucker coordinates of Remark \ref{rmk:flag}.
Following (\ref{eq:wedge23}) and Section \ref{sec3},
each line specifies a linear equation in $Y = \wedge_2 X$
and each plane gives a linear equation in $Z = \wedge_3 X$.

The numerical software {\tt HomotopyContinuation.jl} 
due to Breiding and Timme \cite{notices, julia} is easy to use,
even for those who are not yet familiar with {\tt julia}.
We now go over our steps in solving the system for the
instance in Example \ref{ex:333}. 

The input is a system of $11$ equations in $11$ unknowns, namely the ten
entries of the
matrix $X$ and one more variable $D$.
One equation is $D = \det (X)$, another specifies a random affine chart,
$\sum_{1 \leq i < j \leq 4} c_{ij} x_{ij} = 1$,
and the others are the  tangency conditions.
Our equations are entered into \ {\tt HomotopyContinuation.jl}:
 \begin{figure}[h]
	\centering
	\vspace{-0.1in}
	\begin{small}
		\begin{BVerbatim}
Equations=System(vcat(Point_Conditions,
                      Line_Conditions,
                      Plane_Conditions,
                      det(X)-D, Affine_Chart)) 
 	\end{BVerbatim}
	\vspace{-0.14in}
	\end{small}
\end{figure}

\noindent After entering \ {\tt S=solve(Equations)},
   the following output appears:

\begin{figure}[h]
	\centering
	\vspace{-0.06in}	
	\begin{tiny}
		\begin{BVerbatim}	
Tracking 216 paths... 100% |||||||||||||||| Time: 0:00:11
# paths tracked:                  216
# non-singular solutions (real):  104 (104)
# singular endpoints (real):      84 (83)
# total solutions (real):         188 (187)
		\end{BVerbatim}
	\end{tiny}
\end{figure}
\noindent
This suggests that the program tracked $216 = 1^\alpha  2^\beta 3^\gamma$ 
paths from a total degree start system and that 
it found $104$ real nonsingular solutions.
The variable {\tt S} is a $104$-element array of solutions, each of which is an $11$-element array of floating point numbers. 
The first coordinate is {\tt D}, and the last ten are the coordinates of~{\tt X}.

The following code extracts the $17$-th element of {\tt S} and prints that quadric:

\begin{figure}[h]
	\centering
	\vspace{-0.06in}		
	\begin{small}
		\begin{BVerbatim}
quadric=solutions(S)[17]
@var x[1:4]
Quadric=expand(x'*(X(Equations.variables=>real(quadric)))*x)
\end{BVerbatim}
\end{small}
\begin{tiny}
\begin{BVerbatim}	


-2.974732003076*x2*x1-1.289476735251*x2*x3-10.97658863786*x3*x1+ 
+8.372046844711*x4*x1+8.886907306683*x4*x2+9.704839838537*x4*x3+ 
-5.810893956281*x1^2+2.645663598009*x2^2-5.046922439351*x3^2+0.6937980589394*x4^2
		\end{BVerbatim}
	\end{tiny}
\end{figure}
\noindent These {\tt julia} fragments give a first impression. The details may be found
 at (\ref{eq:ourwebsite}).

\smallskip

One key question about numerical output is
whether it can serve as a mathematical proof. How can we be sure
that the $104$ solutions are indeed solutions and moreover, 
that they are distinct, real, and nondegenerate?
This is addressed by the process of {\em a-posteriori} certification,
which generates an actual~proof.

We carry this out using the Krawczyk method, implemented by
Breiding, Rose and Timme \cite{BRT}. It is based
on interval arithmetic and is now available as a standard
feature in {\tt HomotopyContinuation.jl}.
We note that this implementation
 represents a significant advance
over Smale's $\alpha$-certification that was
used for the $3264$ real quadrics in \cite[Proposition 1]{notices}.
This advance has two aspects.
First, the new method in \cite{BRT} is much faster. Second,
  its output gives a bounding box, allowing us to easily certify that the quadrics are nondegenerate.
 
We now show how certification works for our instance. 
 The input is easy: \vspace{-0.15in}
\begin{figure}[H]
	\centering			
	\begin{small}
		\begin{BVerbatim}
C=certify(Equations,S)
	\end{BVerbatim}
	\end{small}
	\vspace{-0.1in}		
\end{figure}
\noindent The program creates a certificate {\tt C}, and 
it reports on that as follows:
\begin{figure}[H]
	\centering
	\vspace{-0.06in}	
	\begin{tiny}
		\begin{BVerbatim}		
CertificationResult
===================
• 104 solutions given
• 104 certified solutions (104 real)
• 104 distinct certified solutions (104 real)
		\end{BVerbatim}
	\end{tiny}
	\vspace{-0.07in}	
\end{figure}
\noindent The certificate {\tt C} is a list of $104$ lists of $22$ 
intervals $I_1,\ldots,I_{11},J_{1},\ldots,J_{11}$ 
in $\RR$.
The product $B=\prod_{i=1}^{11}(I_i+\texttt{im}\cdot J_i)$ is a box in
$ \mathbb{C}^{11} \simeq \RR^{22}$. That box
  provably contains a unique solution to {\tt Equations}, verified by interval arithmetic.

  Checking that these boxes are disjoint proves that the $104$ solutions are distinct. 
  Checking that $B$ is the only box which intersects the complex conjugate of $B$ itself proves that
  this solution is real.  Checking that $0$ is not contained in $I_1$, the interval for the unknown {\tt D}
   proves that the quadric is nondegenerate. 
   
The following command displays the  certifying box $B$ for the $17$-th quadric:
\begin{figure}[H]
	\centering 
		\vspace{-0.1in}
	\begin{small}
		\begin{BVerbatim}
C.certificates[17].certified_solution
\end{BVerbatim}
\end{small}
\begin{tiny}
\begin{BVerbatim}		


(1.459827495775684e-6 ± 2.2938e-14) + (0.0 ± 2.2938e-14)im
(-0.9684823260468921 ± 1.516e-09) + (0.0 ± 1.516e-09)im
(-0.24789433358973637 ± 2.2975e-11) + (0.0 ± 2.2975e-11)im
(0.44094393300164797 ± 1.1016e-09) + (0.0 ± 1.1016e-09)im
(-0.9147157198219121 ± 1.3088e-09) + (0.0 ± 1.3088e-09)im
(-0.10745639460424983 ± 3.3522e-10) + (0.0 ± 3.3522e-10)im
(-0.8411537398918771 ± 1.0251e-09) + (0.0 ± 1.0251e-09)im
(0.6976705703926359 ± 9.7633e-10) + (0.0 ± 9.7633e-10)im
(0.7405756088903332 ± 1.1508e-09) + (0.0 ± 1.1508e-09)im
(0.8087366532114602 ± 1.202e-09) + (0.0 ± 1.202e-09)im
(0.11563300982325174 ± 7.2217e-10) + (0.0 ± 7.2217e-10)im
		\end{BVerbatim}
	\end{tiny}
\end{figure}

\begin{remark} Finding the fully real instances for  Theorem \ref{thm:inprogress} 
was a challenge.
We implemented a heuristic hill-climbing algorithm similar to the one in \cite{dietmaier}. The idea is to begin at some configuration $\mathcal C$ of $\alpha$ real points, $\beta$ real lines, and $\gamma$ real planes, solve the equations, and sample many nearby instances. If one has more real solutions, then $\mathcal C$ is updated to be that instance. Otherwise, the new $\mathcal C$ is the instance with the same number of real solutions, but whose complex solutions are closest to becoming real. This is measured by the minimum  norm of the complex parts of each nonreal solution. In this fashion, one greedily travels through the parameter space towards instances with  more real solutions. A major issue with such methods is that they get stuck in local maxima. Our success came from many iterations beginning at different randomly chosen parameters. A host of numerical tolerances determine the behavior of this algorithm. Once the number of real solutions approaches the maximum, the instances often become so ill-conditioned that serious monitoring of these tolerances is required.
\end{remark}

\section{Schubert's Pyramid}
\label{sec6}

We now finally come to the analogue in $\PP^3$ of the number $3264$.
The following conjecture motivated this project. We hope that
it can be resolved in the future.

\begin{conjecture} \label{conj:big}
There exist nine quadrics in $\PP^3_\RR$ such that all
$666841088$ complex quadrics that are tangent to these nine
are defined over the real numbers~$\RR$.
\end{conjecture}

We propose a combinatorial gadget for approaching this problem. 
Schubert's pyramid is a tetrahedron of $220$ intersection
numbers $p^\alpha \ell^\beta h^\gamma q^\delta$, where
$(\alpha,\beta,\gamma,\delta) \in \NN^4$ with
$\alpha+\beta+\gamma+\delta = 9$.
Here $q = 2 (p+\ell+h)$ denotes the cohomology class of the complete
quadrics tangent to a given quadric in $\PP^3$.
Thus the pyramid organizes the number of quadrics
tangent to nine figures, as in Figure~\ref{fig:levels}.

\begin{figure}[ht]
  \centering
  \includegraphics[scale=0.42]{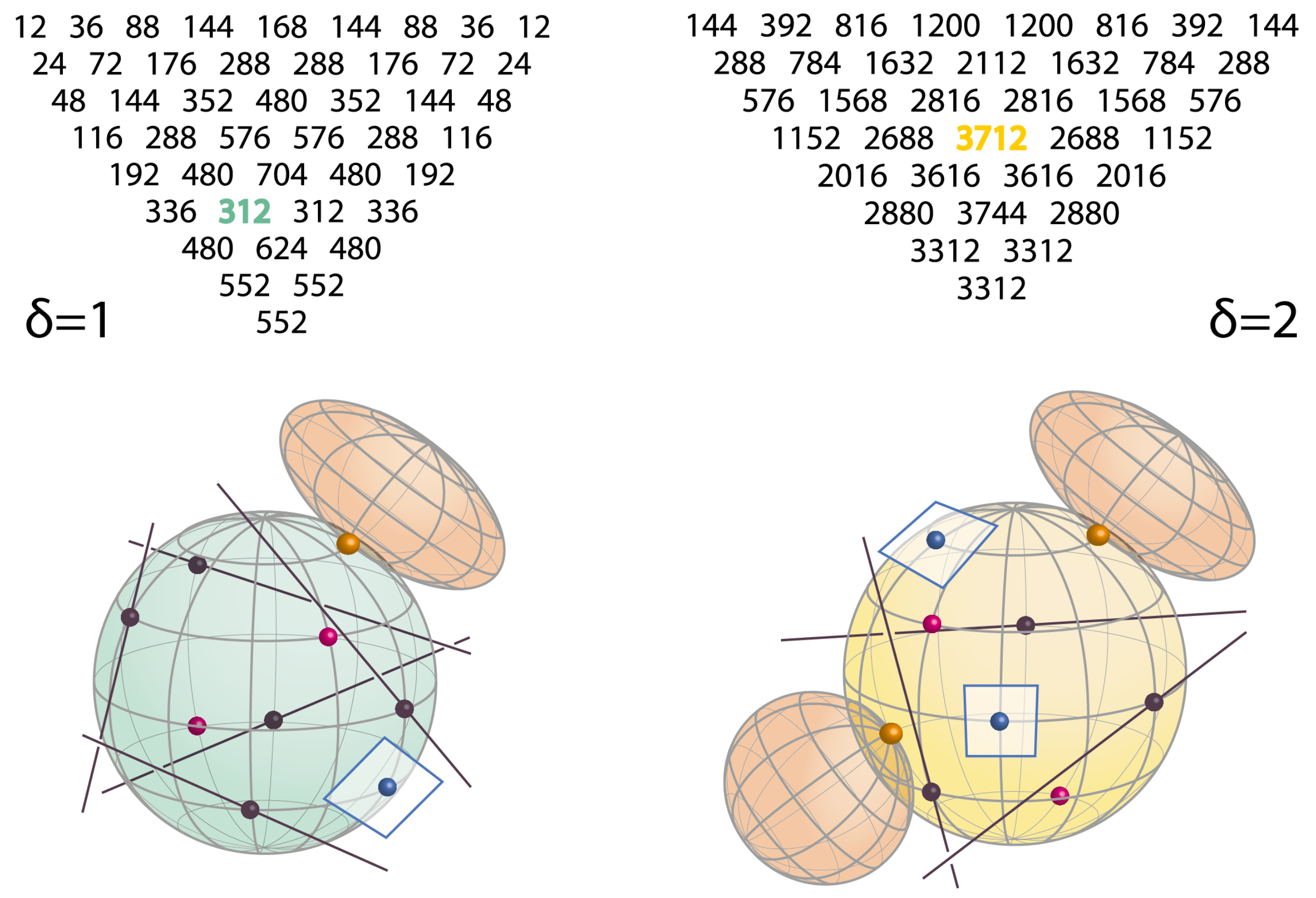}
   \vspace{-0.09in}
  \caption{Two consecutive levels in Schubert's pyramid}
    \label{fig:levels}
\end{figure}

The levels in Schubert's pyramid are the triangles for fixed $\delta$. 
Each entry in level $\delta$ is twice the sum of the three entries
in level $\delta-1$ that lie below it. For instance, for $\delta=2$
we marked $3712 = 2 \cdot (576+576+704)$. This counts~quadrics through two points that are tangent to three lines,
two planes and two quadrics.

Making Schubert's triangle fully real is only a first step towards Conjecture~\ref{conj:big}.
What we really want is to find one single instance of nine real flags:
\begin{equation}
\label{eq:nineflags}
P_1 \subset L_1 \subset H_1\,,\,\,
P_2 \subset L_2 \subset H_2\,,\, \ldots\,, \,\,
P_9 \subset L_9 \subset H_9 . 
\end{equation}
We want those nine flags to exhibit full reality,
simultaneously for all their many tangency problems.
Such a configuration (\ref{eq:nineflags}) would be the 
$3$-dimensional analogue to the
pentagon in \cite[Figure 3]{notices}.
To state this precisely, we consider an arbitrary
function $\psi: \{1,2,\ldots,9\} \rightarrow \{P,L,H\}$. This
defines a polynomial system 
\begin{equation}
\label{eq:manysystems}
 {\rm det}\bigl( \psi(i)_i \cdot X \cdot \psi(i)_i^T\bigr) \,= \, 0  \quad {\rm for}\,\, \,\,i=1,2,\ldots, 9. 
 \end{equation}
This has the form (\ref{eq:schub}),
where $\alpha = |\psi^{-1}(P)|$, $\beta = |\psi^{-1}(L)|$,
and $\gamma = |\psi^{-1}(H)|$. Thus, an instance (\ref{eq:nineflags}) of nine flags
gives a collection of $3^9$ polynomial systems. For each of these, the
number of solutions is one entry in Schubert's triangle.

\begin{conjecture} \label{conj:fullyreal}
There exist nine real flags (\ref{eq:nineflags}) in $\PP^3$ such that
each complex solution $X$ to any of the $3^9$ 
associated polynomial systems (\ref{eq:manysystems}) is a real quadric.
\end{conjecture}

If Conjecture~\ref{conj:fullyreal} is true, then we can
approach Conjecture \ref{conj:big} as follows.
We are given $\,(p+\ell+h)^9 = 1302424$ real quadrics $X$ that solve
the $3^9$ systems. Each solution becomes $2^9$ distinct solutions
 under the deformation in Proposition  \ref{prop:hurwitz}, where
 the nine flags for $\epsilon=0$  become nine smooth quadrics for $\epsilon > 0$.

This process can be performed in stages, from the bottom to the top of the pyramid,
but its numerical implementation will not be easy. One hope is that reality can be
controlled using the results by Ronga, Tognoli and Vust in \cite{ronga}.

\bigskip \bigskip

\noindent {\bf Acknowledgements.} We thank Sascha Timme
for his important contributions.

\begin{small}

\end{small}

\end{document}